\documentclass[12pt]{amsart}
\usepackage{amsmath,amssymb,amsbsy,amsfonts,amsthm,latexsym,mathabx,
            amsopn,amstext,amsxtra,euscript,amscd,stmaryrd,mathrsfs,
            cite,array,mathtools,enumerate}

\usepackage{url}
\usepackage[colorlinks,linkcolor=blue,anchorcolor=blue,citecolor=blue,backref=page]{hyperref}

\usepackage{color}

\renewcommand*{\backref}[1]{}
\renewcommand*{\backrefalt}[4]{%
    \ifcase #1 (Not cited.)%
    \or        (p.\,#2)%
    \else      (pp.\,#2)%
    \fi}

\usepackage{color}
\usepackage{float}

\hypersetup{breaklinks=true}

\usepackage[english]{babel}

\usepackage{mathtools}
\usepackage{todonotes}

\usepackage{enumitem}

\begin{document}

\newtheorem{theorem}{Theorem}
\newtheorem{lemma}[theorem]{Lemma}
\newtheorem{claim}[theorem]{Claim}
\newtheorem{cor}[theorem]{Corollary}
\newtheorem{prop}[theorem]{Proposition}
\newtheorem{definition}{Definition}
\newtheorem{question}[theorem]{Open Question}
\newtheorem{example}[theorem]{Example}

\numberwithin{equation}{section}
\numberwithin{theorem}{section}

 \newcommand{\F}{\mathbb{F}}
\newcommand{\K}{\mathbb{K}}
\newcommand{\D}[1]{D\(#1\)}
\def\scr{\scriptstyle}
\def\\{\cr}
\def\({\left(}
\def\){\right)}
\def\[{\left[}
\def\]{\right]}
\def\<{\langle}
\def\>{\rangle}
\def\fl#1{\left\lfloor#1\right\rfloor}
\def\rf#1{\left\lceil#1\right\rceil}
\def\le{\leqslant}
\def\ge{\geqslant}
\def\eps{\varepsilon}
\def\mand{\qquad\mbox{and}\qquad}

\def\Res{\mathrm{Res}}
\def\vec#1{\mathbf{#1}}

\newcommand{\commA}[1]{\marginpar{%
\begin{color}{red}
\vskip-\baselineskip %raise the marginpar a bit
\raggedright\footnotesize
\itshape\hrule \smallskip A: #1\par\smallskip\hrule\end{color}}}

\newcommand{\commD}[1]{\marginpar{%
\begin{color}{blue}
\vskip-\baselineskip %raise the marginpar a bit
\raggedright\footnotesize
\itshape\hrule \smallskip D: #1\par\smallskip\hrule\end{color}}}

\newcommand{\commI}[1]{\marginpar{%
\begin{color}{magenta}
\vskip-\baselineskip %raise the marginpar a bit
\raggedright\footnotesize
\itshape\hrule \smallskip I: #1\par\smallskip\hrule\end{color}}}

\newcommand{\commL}[1]{\marginpar{%
\begin{color}{green}
\vskip-\baselineskip %raise the marginpar a bit
\raggedright\footnotesize
\itshape\hrule \smallskip L: #1\par\smallskip\hrule\end{color}}}

\def\bl#1{\begin{color}{blue}#1\end{color}} % color text blue during edits

\newcommand{\Fq}{\mathbb{F}_q}
\newcommand{\Fp}{\mathbb{F}_p}
\newcommand{\Disc}[1]{\mathrm{Disc}\(#1\)}

\newcommand{\Z}{\mathbb{Z}}
\renewcommand{\L}{\mathbb{L}}
%\newcommand{\Nm}[1]{\mathrm{Norm}_{\F{q^k/\Fq}}(#1)}

%%%%%%%%%%%%%%%%%%%%%%%%%
% Alphabet calligraphie %
%%%%%%%%%%%%%%%%%%%%%%%%%
\def\cA{{\mathcal A}}
\def\cB{{\mathcal B}}
\def\cC{{\mathcal C}}
\def\cD{{\mathcal D}}
\def\cE{{\mathcal E}}
\def\cF{{\mathcal F}}
\def\cG{{\mathcal G}}
\def\cH{{\mathcal H}}
\def\cI{{\mathcal I}}
\def\cJ{{\mathcal J}}
\def\cK{{\mathcal K}}
\def\cL{{\mathcal L}}
\def\cM{{\mathcal M}}
\def\cN{{\mathcal N}}
\def\cO{{\mathcal O}}
\def\cP{{\mathcal P}}
\def\cQ{{\mathcal Q}}
\def\cR{{\mathcal R}}
\def\cS{{\mathcal S}}
\def\cT{{\mathcal T}}
\def\cU{{\mathcal U}}
\def\cV{{\mathcal V}}
\def\cW{{\mathcal W}}
\def\cX{{\mathcal X}}
\def\cY{{\mathcal Y}}
\def\cZ{{\mathcal Z}}

\def \brho{\boldsymbol{\rho}}

\def \pf {\mathfrak p}

\def \Prob{{\mathrm {}}}
\def\e{\mathbf{e}}
\def\ep{{\mathbf{\,e}}_p}
\def\epp{{\mathbf{\,e}}_{p^2}}
\def\em{{\mathbf{\,e}}_m}

\newcommand{\sR}{\ensuremath{\mathscr{R}}}
\newcommand{\sDI}{\ensuremath{\mathscr{DI}}}
\newcommand{\DI}{\ensuremath{\mathrm{DI}}}

\newcommand{\Nm}[1]{\mathrm{Norm}_{\,\F_{q^k}/\Fq}(#1)}

\def\Tr{\mbox{Tr}}
\newcommand{\rad}[1]{\mathrm{rad}(#1)}

\newcommand{\Orb}[1]{\mathrm{Orb}\(#1\)}
\newcommand{\aOrb}[1]{\overline{\mathrm{Orb}}\(#1\)}

%
%\title[Sparsity of curves over finite fields]
%{Sparsity of curves over finite fields}

\title[Sparsity of curves over finite fields]
{Sparsity of curves and additive and multiplicative expansion
of rational maps over finite fields}

 \author[L. M{\'e}rai]{L{\'a}szl{\'o} M{\'e}rai}
\address{L.M.: Johann Radon Institute for Computational and Applied Mathematics, Austrian Academy of Sciences and Institute of Financial Mathematics and Applied Number Theory,
Johannes Kepler University,  Altenberger Stra\ss e 69, A-4040 Linz, Austria} 
\email{laszlo.merai@oeaw.ac.at}
\author[I.~E.~Shparlinski]{Igor E. Shparlinski}
\address{I.E.S.: School of Mathematics and Statistics, University of New South Wales.
Sydney, NSW 2052, Australia}
\email{igor.shparlinski@unsw.edu.au}

\pagenumbering{arabic}

\begin{abstract} 
For a prime $p$ and a polynomial $F(X,Y)$ over a finite field $\F_p$  of  $p$ elements,  we give upper bounds on the number of solutions
$$
F(x,y)=0, \quad x\in\cA, \ y\in \cB,
$$
where $\cA$ and $\cB$ are very small intervals or subgroups. These bounds  can be considered as positive characteristic analogues of the result of Bombieri and Pila (1989) on sparsity of integral points on curves. 
As an application we prove that 
distinct consecutive elements in sequences generated compositions of several rational functions 
are not contained in any short intervals or small subgroups.
\end{abstract}



\maketitle

\section{Introduction}

\subsection{Background} 
We study some geometric properties of polynomial maps in finite fields.
In particular, we continue investigating the  introduced in~\cite{GoSh15} question of expansion of dynamical systems
generated by 
polynomial and rational function maps in positive characteristic, see~\cite{C2014b,CG2011,CGOS2012,CCGHSZ2014,CSZ2012,IKSSS18,KMS2018+,Kerr,Shp16a,Shp16b,Z2011}
 and the reference therein for recent results, methods and applications. Here  we consider both additive and multiplicative expansion, and also study  more general compositions of several maps. This is based on  new results of sparsity 
 of rational points on algebraic curves over finite fields, which can be considered as an analogue of the 
 celebrated result of Bombieri and Pila~\cite[Theorem~4]{BP1989} on sparsity of integral points on curves
 in characteristic zero. 
 
Let $p$ be a prime number and let $\F_p$ be the finite field of $p$ elements, represented by the integers $\{0,1,\ldots p-1\}$. For a polynomial $F\in\F_p[X,Y]$  and sets $\cA,\cB\subset \F_p$ write
$$
 N_F(\cA,\cB)=\#\{(a,b)\in\cA\times\cB~:~F(a,b)=0 \}.
$$

Our goal is to give bounds on $N_F(\cA,\cB)$ for  some interesting sets such as intervals or subgroups
and in particular, improve the trivial bound
$$
 N_F(\cA,\cB)  = O\( \min\{\#\cA, \# \cB\}\). 
$$ 
We are especially interested in the case of the sets of small cardinalities $\#\cA$ and $\# \cB$,  more 
specifically in the cases where traditional methods coming from algebraic geometry do not work anymore.

If  both $\cA$ and $\cB$ are intervals of length $H$ and $F(X,Y)$ is absolutely irreducible, then it is known from 
the Bombieri bound~\cite{Bom66} 
that
$$
N_F(\cA,\cB)=\frac{H^2}{p}+O(p^{1/2}(\log p)^2),
$$
where the implied constant depends only on $\deg F$, see~\cite{GSZ2005}. The main term dominates the error term if $H\ge p^{3/2}\log p$ and for $H\le p^{1/2}(\log p)^2$ the result becomes weaker than the trivial upper bound $N_F(\cA,\cB) = O(H)$. For smaller $H$ and for polynomials $F$ having a special form,
this question have been studied intensively in the literature,
%bounds on $N_F(\cA,\cB)$ have been give, 
see~\cite{C2014b,CG2011,CGOS2012,CCGHSZ2014,CSZ2012,Kerr,Z2011} and the reference therein. For example, when $F$ defines a modular hyperbola,  $F(X,Y)=X\cdot Y-c$ for some $c\neq 0$, the problem is studied by Cilleruelo and Garaev~\cite{CG2011}, see also~\cite{Z2011}, for general quadratic forms. If $F$ defines the graph of a polynomial $F(X,Y)=Y-f(X)$ or an hyperelliptic curve $F(X,Y)=Y^2-f(X)$, the problem was studied by  Cilleruelo, Garaev,  Ostafe and Shparlinski~\cite{CGOS2012} and by 
Chang,  Cilleruelo,  Garaev,  Hern\'andez,  Shparlinski,  Zumalac\'arregui~\cite{CCGHSZ2014}.
Finally, Chang~\cite{C2014b} considered the function of the form $F(X,Y)=f(X)-g(Y)$.

Here we consider the general case of arbitrary bivariate polynomials $F(X,Y)$. Althought some 
of our techniques have already been used, the main novelty of this paper is in investigating the conditions
under which these ideas apply and also in exploiting the possible sparsity of the  polynomial $F$, see the definition 
of $\delta(F)$ below. We also give new application to the dynamics of polynomial semigroups.

\subsection{Notation} 

In order to state the result we denote by $\delta(F)$ the number of 
distinct divisors of the monomial terms of 
$$F(X,Y)=\sum_{(i,j)\in \cF }F_{i,j}X^iY^j$$
where $\cF$ is the support of the coefficients of $F$, that is,
$F_{i,j} \ne 0$ if and only if $(i,j) \in \cF$. 
Alternatively
$$
\delta(F)=\#\{(k,\ell)~:~0\le k\leq i, \ 0 \le \ell \leq j \text{ for some } (i,j) \in \cF\}.
$$
Clearly, we have the following trivial bound
$$
\delta(F)\leq \binom{\deg F+1}{2},
$$
which is attained for dense polynomials. However for sparse polynomials it can be significantly 
smaller; for example $\delta(X^n + Y^n + XY) = 2n+2$ for any $n \ge 1$. 

As usual, we use $\deg_X F$ and $\deg_Y F$ for the local degrees of $F$ with respect to 
$X$ and $Y$,  respectively. 

We recall that the
notations $U = O(V)$, $U \ll V$ and  $V \gg U$ are all equivalent to the
statement that the inequality $|U| \le c V$ holds with some absolute
constant $c> 0$. We also write $U\ll_{d} V$ or $U\ll_{d,\nu}V$ if the implied constants may depend on $d$ or on $d$ and $\nu$.

\subsection{Main results} 

We start with the case of intervals

\begin{theorem}\label{thm:II}
Let $F(X,Y)\in\F_p[X,Y]$ be an absolutely irreducible polynomial of degree $d\geq 2$ with  $\delta=\delta(F)$. Then for any positive integer 
 $$
 H\leq p^{1/ \left( (d-1/2)\delta +1/2 \right) }, 
 $$
 uniformly over arbitrary intervals $\cI=[K,K+H]$ and $\cJ=[L,L+H]$, we have 
$$
  N_F(\cI,\cJ)\le H^{1/d+o(1)}.
$$
\end{theorem}

For a rational function $\psi(X)=f(X)/g(X)$ with $f(X),g(X)\in\F_p[X]$ and for a set $\cA\subset \F_p$ write
$$
\psi(\cA)=\{\psi(x)~:~\ x\in\cA, g(x)\neq 0 \}.
$$
Applying the result to $F(X,Y)=f(X)-Yg(X)$ with $\delta(F)\leq 2 \deg F +2$, we have the following bound on values of the rational function $\psi(X)$ in small intervals.

\begin{cor}\label{cor:psi}
Let $\psi(X)\in\F_p(X)$ be a rational function of degree $d\geq 2$. Then for any positive integer 
$$
 H\leq p^{1/ ( 2d^2+d-1/2 ) }, 
 $$
uniformly over arbitrary intervals $\cI=[K,K+H]$ and $\cJ=[L,L+H]$, we have 
$$
  \# (\psi(\cI)\cap \cJ) \leq H^{1/d+o(1)}.
$$
\end{cor}

Here we also consider the case, when $\cA$ or $\cB$ is a subgroup. We recall that for subgroups $\cG,\cH\subset \F_p^*$, Corvaja and Zannier~\cite{CZ2013} haven give a  nontrivial on $N_F(\cG,\cH)$. Using their result we obtain a bound on $N_F(\cI,\cG)$ with an interval $\cI$ and a subgroup $\cG$. It extends the result of Karpinski,  M{\'e}rai and Shparlinski~\cite{KMS2018+} who bound $N_F(\cI,\cG)$ with $F(X,Y)=f(X)-Yg(X)$, see also~\cite{GoSh15,IKSSS18,Shp16a,Shp16b}.

\begin{theorem}\label{thm:IG}
 Let $F(X,Y)\in\overline{\F}_p[X,Y]$ be a polynomial of total degree $d$ and local degrees $d_X , d_Y  \geq 1$ such that $F(X,Y^n)$ is irreducible for all $n\geq 1$. Then for any interval $\cI=[1,H]$ of length $H<p$ and any subgroup $\cG\subset\overline{\F}_p^*$ of order $e$, we have
 $$
  N_F(\cI,\cG)\ll d_X^{1/2} H^{1/2}\max\{d e   p^{-1/2},d^{2/3}e^{1/3}, d_X^{1/2}d_Y^2\}.
 $$
\end{theorem}

\subsection{Applications}

As an application of these results we study the geometric properties of the orbits of a transformation $x\mapsto \psi(x)$ associated with a rational function $\psi(X)\in\F_p(X)$, which have also being studied in~\cite{C2012,C2014a,CGOS2012,GS2010}. 
In fact we consider a much more general scenarios of  semigroups generated by  a system 
$$\varPsi = \( \psi_{1}, \ldots, \psi_{s}\)  \in\F_p(X)^s
$$
of $s$ rational functions.

 More precisely, for $u\in\F_p$, we call any sequence
 \begin{equation}\label{eq:u-seq}
u_0=u, \quad u_{n+1}=\psi_{j_n}(u_n), \ n=0,1,\ldots,   
\end{equation}
with arbitrary $j_n \in \{1, \ldots, s\}$
a {\it path originating from $u$\/}. Let $\Pi_{\Psi, u}$ be the set of possible paths
originating from $u$. 
We denote by $T_{\Psi, u}$   the largest positive integer $T$ for which in any path $\{u_n\} \in \Pi_{\Psi, u}$ the first $T$ elements are pairwise distinct. 
Clearly if $s=1$ then  $T_{\Psi, u}$ is the total length of the pre-periodic and periodic 
parts of the orbit of $u$ in the dynamical system generated by $\Psi$.

Given $u\in\F_p$, we consider the sequence $(u_n)$ as a dynamical system on $\F_p$ and study how far it propagates in $N$ steps or how large is the  group the first $N$ elements generate along its paths. 
In fact, we study more generals quantities
$$
L_{\Psi, u}(N)=\min_{\{u_n\} \in \Pi_{\Psi, u}}\min_{v \in \F_p} \max_{0\leq n\leq N}  |u_n-v|
$$
and
$$
G_{\Psi, u}(N)= \min_{\{u_n\} \in \Pi_{\Psi, u}} \min_{v \in \F_p^*} \#\langle  vu_n~:~n=0,\ldots, N \rangle. 
$$

Previously these quantities have been studied only for $s=1$, here we show that our methods work as well for an arbitrary $s$. 

It has been shown by Gutierrez and Shparlinski~\cite{GS2010}, in the case of one function, that is, for $\Psi = \left(\psi\right)$, that
$$
L_{\Psi, u}(N)=p^{1+o(1)}
$$
provided that $N\geq p^{1/2+\varepsilon}$ for some fixed $\varepsilon > 0$. Moreover, for linear fractional function $\psi(X)=(aX+b)/(cX+d)$ with $ad\neq bc$ and $c\neq 0$ they obtain nontrivial result for shorter orbit, namely
$$
L_{\Psi, u}\gg N^{1+\delta}
$$
provided $N\leq \min\{T_{\Psi,u}, p^{1-\varepsilon}\}$. In the case of one  polynomial,  Cilleruelo, Garaev,  Ostafe and Shparlinski~\cite{CGOS2012} obtain a lower bound for essentially
arbitrary values of $N$ and $T_{\Psi,u}$.

Here, to exhibit the main ideas in the least technical way, we always assume that the functions
$\psi_{1}, \ldots, \psi_{s}\  \in\F_p(X)$ are of the same degree. Note that this condition automatically 
holds in the classical case of one functions. 

\begin{theorem}\label{thm:orbit-L}
Let 
$$\varPsi = \( \psi_{1}, \ldots, \psi_{s}\)  \in\F_p(X)^s
$$  
be  rational functions of degree $d\geq 2$. Then for  $N\leq T_{\Psi,u}$  
and any integer $\nu\ge 1$ we have
$$
L_{\Psi, u}(N)\gg_{s,d,\nu} \min\left\{ N^{d^\nu+o(1)}, p^{1/(2d^{2\nu} +d^{\nu}-1/2)}\right\}.
$$ 
\end{theorem}

The group size $G_{\Psi, u}(N)$ generated by the first $N$ elements has been investigated~\cite{GoSh15,Shp16a,Shp16b}. For example, in~\cite[Theorem~1.1]{Shp16a} it is proved that for a polynomial $f(X)\in\Fq[X]$ of degree $d$, satisfying some mild conditions, and for initial value $u\in\Fq$, we have for any $\nu\geq 1$ that
$$
G_{\Psi, u}(N)\geq C_1(\nu,d) N^{1/\nu}q^{1-1/\nu},\qquad \text{for } T_{\Psi,u}\geq N\geq C_2(\nu,d)q^{1/2},
$$
for some constants $C_1(\nu,d), C_2(\nu,d)$ which may depend only on $\nu$ and $d$. Shparlinski~\cite{Shp16a} has also given a lower bound on $G_{\Psi, u}(N)$ for much smaller $N$ under the condition that the coefficients of the polynomial are small. Namely, if the polynomial $f(X)$ is defined over $\mathbb{Z}$, then for any prime $p$ and any initial value $u\in\F_p$ we have
$$
G_{\Psi, u}(N)\geq \sqrt{\frac{(N-2d)\log p}{\log((d+2)h)}},  \qquad \text{for } T_{\Psi,u}\geq N\geq 2d,
$$
where $d$ is the degree and $h$ is the  
height of $f(X)$, respectively, that is $h=\max\{|a_0|,\ldots, |a_d|\}$ where $f(X)=a_nX^n+\ldots + a_1X + a_0$.

Here we bound $G_{\Psi, u}(N)$ for rational functions $\varPsi = \left( \psi_{1}, \ldots, \psi_{s}\right)  \in\F_p(X)^s$.  
If any of the rational function $\psi_i$ is a power, that is $\psi_i(X)=\alpha X^{m}$, $m\in\mathbb{Z}$, then there are paths contained in small subgroup. Indeed, if $\alpha, u\in \cG$  for a small subgroup $\cG$, then the path $\{u_n\}$ is in $\cG$ with the choice $j_n=i$, $n=1,2,\ldots$. However, if all the rational functions $\psi_i$ are not powers, we can give nontrivial bound on $G_{\Psi,u}$. Note that here we do not need the simplifying assumption that all functions  
$ \psi_{1}, \ldots, \psi_{s}$ are of the same degree.
 
\begin{theorem}\label{thm:orbit-G}
Let 
$$\varPsi = \( \psi_{1}, \ldots, \psi_{s}\)  \in\F_p(X)^s
$$
with
$$
\psi_i(X)\neq \alpha X^m \quad  \alpha\in \F_p, \ m\in\mathbb{Z}, \ i=1,\ldots, s,
$$
be  rational functions of degree at most $d\geq 2$. Then for  $N\le T_{\psi,u}$
$$
 G_{\Psi, u}\gg \min\left\{\frac{N^{1/2}p^{1/2}}{ds^{1/2}}, \frac{ N^{3/2}}{d^2s^{3/2}}\right\}.
$$
\end{theorem}

\section{Proof of Theorem~\ref{thm:II}}

\subsection{Preliminaries}

The proof  is based on the following bound of Bombieri and Pila~\cite[Theorem~4]{BP1989}.

\begin{lemma}\label{lemma:BV}
 Let $\cC$ be an absolutely irreducible curve of degree $d\geq 2$ and $H\geq \exp(d^6)$. Then the number of integral points on $\cC$ and inside of a square $[K,K+H]\times[L,L+H]$ does not  exceed 
 $$H^{1/d}\exp(12 \sqrt{d \log H \log \log H}).$$
\end{lemma}

We recall, that a polynomial $F(X,Y)\in\mathbb{Z}[X,Y]$ is said to be {\it primitive} if the greatest common divisor of its coefficients is 1. 

\begin{lemma}\label{lemma:irred}
Let $F(X,Y)\in\mathbb{Z}[X,Y]$ be a primitive polynomial of degree $d=\deg F$ and let $p$ be a prime number.
Let $\widebar{F}(X,Y)\in\F_p[X,Y]$ with $F(X,Y)\equiv \widebar{F}(X,Y)\mod p$. Assume, that $\widebar{F}(X,Y)$ has degree $d$ and is absolutely irreducible. Then $F(X,Y)$ is absolutely irreducible.
\end{lemma}

\begin{proof}
 Assume that $F(X,Y)$ is not absolutely irreducible, that is there is a number field  $\K$ such that
 $$
 F(X,Y)=f(X,Y)\cdot g(X,Y),\quad  f(X,Y),g(X,Y)\in\K[X,Y].
 $$
 By the multivariable Gauss lemma~\cite{Gauss} we can assume, that $f,g\in\cO_\K[X,Y]$ with $\deg f, \deg g<d$, where $\cO_\K$ is the ring of integers of $\K$.
 Let $\pf$ be a prime ideal in $\cO_\K$ over $p$. Then $\cO_\K/\pf$ is a finite extension of $\F_p$ and $\widebar{F}(X,Y)\equiv f(X,Y)\cdot g(X,Y) \mod \pf$. As the degrees of $f(X,Y)$ and $g(X,Y)$ modulo  $\pf$ are
 strictly less than $d$, we get a nontrivial factorisation of $\widebar{F}(X,Y)$.
\end{proof}

\subsection{Concluding the proof} 
Put $N=N_F(\cI,\cJ)$. Replacing $F(X,Y)$ by $F(X-K_0,Y-L_0)$ for some $K_0,L_0$, we can assume that the congruence
$$
  F(x,y)\equiv 0 \mod p, \quad  |x|,|y|\leq H, 
$$
has at least $N$ solutions. 
Covering the square $[-H,H]^2$ by $O(N)$ squares  with the side length
$2dH/\sqrt{N}$ and shifting the variables again, we can also assume, that there are solutions
\begin{equation}\label{eg:small_points}
(x_1,y_1),\ldots, (x_{d^2+1},y_{d^2+1}), \qquad 0<x_k, y_k\leq 2dH/\sqrt{N}.
\end{equation}
We remark, that the quantity $\delta(F)$ is invariant under linear changes of variable, that is $\delta(F(X,Y))=\delta(F(X-K_0,Y-L_0))$.

Let $\Delta$ be the set of multiindices $(i,j)$ such that the coefficient of $X^iY^j$ in $F(X,Y)$ is nonzero. By definition we have $\#\Delta\leq \delta$.

Write 
$$
F(X,Y)=\sum_{(i,j)\in \Delta}  F_{i,j}X^iY^j
$$
and consider the linear congruence system 
\begin{equation}\label{eq:system}
  \sum_{(i,j)\in \Delta}  F_{i,j}x_k^iy_k^j\equiv 0 \mod p, \qquad k=1,\ldots, d^2+1, 
\end{equation}
for the coefficients $F_{i,j}$. This system determines $F(X,Y)$ up to a constant multiple. Indeed, if  $G(X,Y)$ is a polynomial of degree at most $d$ whose non-zero coefficients satisfy the system of congruences~\eqref{eq:system}, then $G(x_k,y_k)=0$ for $k=1,\ldots, d^2+1$, 
%% thus 
then 
$$
\#\{F(X,Y)=G(X,Y)=0\}> d^2.
$$
As $F(X,Y)$ is absolutely irreducible, by the B\'ezout theorem we have $F(X,Y)\mid G(X,Y)$. Thus the coefficient matrix of~\eqref{eq:system} of size $(d^2+1)\times \delta$ has rank $\delta-1$ over $\Fp$. We can assume that the first $\delta-1$ rows are linearly independent over $\Fp$ and thus over $\mathbb{Q}$.

Fix $(k,\ell)\in \Delta$ and let $V\in\mathbb{Z}^{(\delta-1)\times (\delta-1)}$ be a matrix whose columns are indexed by the elements $\Delta\setminus\{(k,\ell)\}$ and the $(m,n)$-th column is
$$
V_{(m,n)}=
\left(
x_1^my_1^n,\ldots,x_{\delta-1}^my_{\delta-1}^n
\right)^T.
$$
Similarly, for $(i,j)\neq (k,\ell)$, let $U(i,j)\in\mathbb{Z}^{(\delta-1)\times (\delta-1)}$ be a matrix  such that its $(m,n)$-th column is
$$
U(i,j)_{(m,n)}=
\left\{
\begin{array}{cl}
 V_{r,s} & \text{if } (r,s)\neq (i,j),\\
 \left(-x_1^ky_1^\ell,\ldots,-x_{\delta-1}^ky_{\delta-1}^\ell\right)^T& \text{if } (r,s)= (i,j).
\end{array}
\right.
$$

Put
\begin{equation}\label{eq:det}
v=\det V \quad \text{and} \quad u_{i,j}=\det U_{i,j}, \ (i,j)\neq (k,\ell).
\end{equation}
Then $v\neq 0$ and by the Cramer rule
$$
 v \cdot F_{i,j}\equiv u_{i,j} \cdot F_{k,\ell}\mod p \quad \text{for} \quad  (i,j)\neq (k,\ell)
$$
and by~\eqref{eg:small_points} and~\eqref{eq:det} we have
$$
 |v|, |u_{i,j}|\leq \delta! \ (2dH/\sqrt{N})^{d(\delta-1)}.
$$

Write 
$$
\bar{F}(X,Y)=  \sum_{(i,j)\in \Delta \setminus \{ (k,\ell)\}} u_{i,j}X^iY^j +vX^kY^\ell  \in\mathbb{Z}[X,Y],
$$
then the congruence $F(x,y)\equiv 0 \mod p$, $ |x|,|y|\leq H$,  is equivalent to the congruence
$$
\bar{F}(x,y)\equiv 0 \mod p, \quad |x|,|y|\leq H.
$$
We can write  it as the diophantine equation
\begin{equation}\label{eg:diophantine}
 \widebar{F}(x,y)= pt, \quad  |x|,|y|\leq H , \quad t\in\mathbb{Z}.
\end{equation}

All possible values $t$ satisfies
\begin{align*}
 |t|& \leq \frac{ 1 }{p}\left( \sum_{(i,j)\in \Delta \setminus \{ (k,\ell)\}} |u_{i,j}| H^{i+j}  +|v|H^{k+\ell}\right) 
 \ll_d\frac{(H/\sqrt{N})^{d(\delta-1)} H^d }{p}\\
 &\ll_d \frac{(H/\sqrt{N})^{d(\delta-1)} H^d}{H^{\left( (d-1/2) \delta +1/2 \right)}   }
 =\frac{H^{(\delta-1)/2}}{N^{d(\delta-1)/2}}.
\end{align*}

For each value of $t$ the polynomial $\widebar{F}(X,Y)-pt$ is absolutely irreducible by Lemma~\ref{lemma:irred}, thus by Lemma~\ref{lemma:BV}, the number of integral solutions of~\eqref{eg:diophantine} inside the box $[-H,H]\times[-H,H]$ is at most $H^{1/d+o(1)}$, thus
$$
N\le  \left( \frac{H^{(\delta-1)/2}}{N^{d(\delta-1)/2}}   \right)  H^{1/d+o(1)}
$$
which implies $N\le H^{1/d+o(1)}$ as $H\to \infty$.

\section{Proof of Theorem~\ref{thm:IG}}

\subsection{Preliminaries} 
The proof of the theorem is based on the result of Corvaja and Zannier~\cite[Corollary~2]{CZ2013} in a form given by Karpinski, M{\'e}rai and Shparlinski~\cite{KMS2018+}.

\begin{lemma}\label{lemma:KS}
 Assume that $F(X,Y)\in\overline{\F}_p[X,Y]$ is of degree $\deg F=d$ and does not have the form 
 \begin{equation}\label{eq:torsion}
  \alpha X^m Y^n+\beta \qquad \text{or} \qquad \alpha X^m +\beta Y^n.
 \end{equation}
For any multiplicative subgroups $\cG,\cH\subset \overline{\F}_p^*$, we have
$$
 N_F(\cG,\cH)\ll \max\left\{\frac{d^2W}{p}, d^{4/3}W^{1/3} \right\},
$$
where 
$$
 W=\#\cG\#\cH.
$$
\end{lemma}

We need the following result on the non-vanishing of some resultant.

\begin{lemma}\label{lemma:resultant}
Let $F(X,Y)\in\overline{\F}_p[X,Y]$ such that $F(X,Y^n)$ is irreducible for all $n\geq 1$ with 
local degrees $d_X , d_Y  \geq 1$. Then there is a set $\cE\subset \overline{\F}_p^*$ of cardinality $\#\cE \ll (d_Xd_Y )^2$ such that for $a\in\overline{\F}_p^*\setminus\cE$ the resultant
$$
R_a(U,V)=\Res_X\(F(X,U),F(X+a,V)\)
$$
with respect to $X$, is not divisible by a polynomial of the form~\eqref{eq:torsion}.
\end{lemma}

\begin{proof}
If $R_a(U,V)$ is divisible by a polynomial of form~\eqref{eq:torsion}, then the variety
\begin{equation}\label{eq:variety}
F(X,Y^m)=0 \mand  F(X+a,bY^n)=0
\end{equation}
has positive dimension for some $b$ and some integers $m,n$. 

If $mn=0$ the result is trivial. Replacing $F(X,Y)$ by $Y^{d_Y }F(X,1/Y)$, we can assume that $m>0$. 

First assume, that $n>0$. As both polynomials are irreducible, they are equal up 
to a constant factor  $\alpha \in \overline{\F}_p^*$:
$$
F(X,Y^m)=\alpha F(X+a,bY^n)
$$
and thus $m=n$.
Write
$$
F(X,Y)=\sum_{i=0}^{d_Y} f_i(X)Y^i .
$$
Let $h$ be the maximal index such that $f_h(X)$ is not constant.  Then
$$
f_h(X)Y^{hm}=\alpha f_{d_Y}(X+a)b^hY^{hn},
$$
so there are at most $\deg f_h(X)\leq d_X  $ choices for $a$ and $h\leq d _Y $ choices for $b$.

If $n<0$, then~\eqref{eq:variety} is equivalent to
$$
F(X,Y^m)=0 \mand (bY^n)^{d_Y}F(X+a, bY^{-n})=0, 
$$
and we get in the same way as before that there are at most $ d_X$ choices for $a$ and $d_Y$ choices for $b$.

As $m\leq \deg_U R_a \ll d_X$ and $n\leq \deg_V R_a\ll d_Y$ we get the result.
\end{proof}

\subsection{Concluding the proof}

We  now closely follow the  proof of~\cite[Lemma~5.1]{KMS2018+}.

Denote $N=N_F(\cI,\cG)$. Let $\overline{\cI}=[-H,H]$. Then the system of equations
$$
 F(x,u)=0, \quad F(x+y,v)=0, \quad x\in\cI, y\in\overline{\cI}, u,v\in\cG
$$
has at least $N^2$ solutions. Let $N_y$ be the number of solutions with a fixed $y$. Then
$$
 \sum_{y\in\overline{\cI} }N_y \geq N^2.
$$
Let
$$
 L=\frac{N^2}{2(2H+1)}
$$
and write
$$
 \cY=\{y\in\overline{\cI}~:~N_y>L\}.
$$
Using that $N_y\leq d_Y^2 H$ we write 
$$
 d_Y^2 H \#\cY\geq \sum_{\substack{y\in\overline{\cI}\\ N_y>L }}N_y\geq N^2-\sum_{\substack{y\in\overline{\cI} \\ N_y\leq L }}N_y\geq N^2-(2H+1)L\geq \frac{1}{2}N^2.
$$
Let $\cE$ as in Lemma~\ref{lemma:resultant}. If $\#\cY\leq \#\cE$, then 
$$
N^2\leq 2H\#\cY\ll d_X^2d_Y ^4  H
$$ 
so $N\ll d_X d_Y^2  \sqrt{H}$. Thus we can assume, that $\#\cY\geq \#Y$. Then fix an $a\in\cY\setminus \cE$ and consider the system of equations
$$
F(x,u)=0 \quad \text{and} \quad F(x+a,v)=0, \quad x\in\cI, u,v\in\cG. 
$$
Then for $R_a(U,V)=\Res_X\left(F(X,U), F(X+a,V)\right)$, we have $R_a(u,v)=0$ for each solution $u,v$. 
By Lemma~\ref{lemma:resultant}, $R_a(U,V)$ is not divisible by a polynomial of form~\eqref{eq:torsion}, thus by Lemma~\ref{lemma:KS} we have
\begin{align*}
L&<N_a\leq d _X F \#\{(u,v)\in\cG\times \cG~:~R_a(u,v)=0\}\\
 &\ll d_X   \max\left\{\frac{d ^2e^2}{p}, d^{4/3}e^{2/3} \right\}
\end{align*}
as for a fixed $u$ there are at most $d_X$ possible values for $x$. Recalling the definition of $L$, we obtain
$$
N\ll d_X^{1/2} H^{1/2}\max\{de   p^{-1/2},d^{2/3}e^{1/3}   \}, 
$$
which concludes the proof.

\section{Proof of Theorem~\ref{thm:orbit-L}} 

Consider a path $\{u_n\} \in \Pi_{\Psi, u}$ and assume that $\psi_{j_1}, \ldots\psi_{j_N}$ 
is the sequence of rational functions used to generate the first $N$ elements $u_1, \ldots, u_N$ as in~\eqref{eq:u-seq}. 
Let $\psi_{i_\nu},\ldots, \psi_{i_1}$ 
is the most frequent sequence 
of rational functions among  $\psi_{j_{n+\nu-1}},\ldots, \psi_{j_n}$ for $n = 0, \ldots,  N-\nu$, 
and let $\cN$ be the set of corresponding values of $n$. 

Clearly,  $\cN$ is of cardinality at least 
\begin{equation}
\label{eq:large N}
\# \cN \ge \frac{N-\nu}{s^\nu}. 
\end{equation}

Denote by $\psi$ the composition 
$$
\psi(X) = \psi_{i_\nu}\(\ldots \(\psi_{i_1}(X)\)\).
$$

Since 
$$ \cN \subseteq [1,N-\nu]  \subseteq [1, T_{\Psi,u}-\nu],
$$
the pairs $(u_n,u_{n+\nu})=(u_n,\psi^{(\nu)}(u_n))$ are all pairwise distinct for  $n \in \cN$ and belong of the square 
$$[v-L_{\Psi, u}(N),v+L_{\Psi, u}(N)]\times[v-L_{\Psi, u}(N),v+L_{\Psi, u}(N)].
$$ 
Thus by Corollary~\ref{cor:psi} applied to the iteration $\psi^{(\nu)}$ we have
$$
\cN\le  L_{\Psi, u}(N)^{1/d^\nu+o(1)}
$$
as $N\to \infty$, and recalling~\eqref{eq:large N} we conclude the proof.

\section{Proof of Theorem~\ref{thm:orbit-G}}

Let $\cG$ be the group generated by $\{vu_n:\ n=0,\ldots, N\}$.
We now proceed as in the proof  of Theorem~\ref{thm:orbit-L} with $\nu =1$. Namely, 
let $\psi$ be the most frequent function among   $\psi_{j_n}$ for $n = 1, \ldots,  N-\nu$, 
and let $\cN$ be set of corresponding values of $n$. 

Then $\left(v u_n,v u_{n+1}\right)=\left(vu_n,v\psi(u_n)\right)\in  \cG\times \cG$ for   $n \in \cN$ thus by 
Lemma~\ref{lemma:KS} and~\eqref{eq:large N}  we have
$$
N/s\ll \max\left\{\frac{d^2\#\cG^2}{p}, d^{4/3}\#\cG^{2/3}\right\}
$$
and the result follows.

\section*{Acknowledgement}

The research of L.M. was is supported by the Austrian Science Fund (FWF): Project I1751-N26, and  of I.S. was  supported in part by the Australian Research Council Grants DP170100786 and DP180100201.

\end{document}